\documentclass[11pt]{amsart}
\usepackage[T1]{fontenc}
\usepackage[utf8]{inputenc}
\usepackage{lmodern}
\usepackage[a4paper,margin=32mm]{geometry}
\usepackage{amsmath,amssymb,amsfonts,amsthm,mathtools}
\usepackage{enumitem}
\usepackage{hyperref}

\hypersetup{colorlinks=true,linkcolor=blue,citecolor=blue,urlcolor=blue}
\numberwithin{equation}{section}

\theoremstyle{plain}
\newtheorem{theorem}{Theorem}[section]
\newtheorem{proposition}[theorem]{Proposition}

\newtheorem{corollary}[theorem]{Corollary}

\newtheorem{question}[theorem]{Question}
\newtheorem{remark}[theorem]{Remark}
\email{zazagoginava@gmail.com}
\email{ugoginava@uaeu.ac.ae}
\subjclass[2020]{Primary 42C10; Secondary 42A24, 46E30}
\keywords{Walsh--Paley system, de la Vall\'ee Poussin means, everywhere divergence, Orlicz spaces, almost everywhere convergence}

\begin{document}
\title[de la Vall\'ee Poussin Means ]{de la Vall\'ee Poussin Means of Walsh-Fourier Expansions}
\author{Ushangi Goginava}
\address{Department of Mathematical Sciences, United Arab Emirates
University, P.O. Box 15551, Al Ain, Abu Dhabi, UAE}

\begin{abstract}
We study de la Vall\'ee Poussin means of Walsh--Fourier series associated
with a nondecreasing window sequence. We establish a sharp criterion for
almost everywhere convergence for integrable functions. We further show
that, when this criterion fails, every Orlicz class below the logarithmic
square-root scale contains a function whose de la Vall\'ee Poussin means
diverge everywhere.
\end{abstract}

\maketitle


\section{Introduction}

Questions on convergence and divergence of Fourier series in function spaces
remain central in harmonic analysis. For trigonometric Fourier series,
Kolmogorov proved the existence of an integrable function with almost
everywhere divergent Fourier series \cite%
{Kolmogorov-ae,Kolmogorov-everywhere}, while Marcinkiewicz constructed an
example with bounded divergence almost everywhere \cite{Marcinkiewicz}.
Further divergence results in Orlicz classes were obtained by Prokhorenko 
\cite{Prokhorenko}, Tandori \cite{Tandori-divergence}, Lukashenko \cite%
{Lukashenko}, K\"orner \cite{Korner-sets,Korner-everywhere}, Gosselin \cite%
{Gosselin}, Totik \cite{Totik}, Konyagin \cite%
{Konyagin-everywhere,Konyagin-subsequences,Konyagin-ICM}, Heladze \cite%
{Heladze}, G\'at \cite{gat}. The classical problem posed by Zygmund asks
whether the Orlicz class generated by the function $t\mapsto t\log^{+}t$
guarantees almost everywhere convergence of Fourier series; this is still
open \cite[Ch. XIII]{Zygmund}.

Parallel problems have been investigated for the Walsh--Paley system. Stein 
\cite{Stein} proved the existence of an integrable function whose
Walsh--Fourier series diverges everywhere. Divergence in Orlicz classes
close to $L^{1}$ was established by Moon \cite{Moon}, Schipp \cite%
{Schipp-divergence}, and Simon \cite{Simon}; see also Bochkarev, \cite%
{Bochkarev} Konyagin \cite{Konyagin-Walsh}, G\'at, Goginava, Karagulyan,
Mukhamedov, Oniani \cite{GGK-Div,GogProcDiv,GogJMAA-Div,GogOnDiv}. We refer
to the monograph of Schipp, Wade, and Simon \cite{Schipp-Wade-Simon} for
background on dyadic harmonic analysis.

Let $\lambda=\{\lambda_n\}_{n\ge1}$ be a nondecreasing sequence of integers
satisfying 
\begin{equation}  \label{eq:lambda-basic}
1\le \lambda_n \le n \qquad (n\in\mathbb{N}).
\end{equation}
For $f\in L^{1}([0,1))$ we define the de la Vall\'ee Poussin means of the
Walsh--Fourier series of $f$ by 
\begin{equation}  \label{eq:vp-def}
V_n^{(\lambda)}(f;x):=\frac{1}{\lambda_n+1}\sum_{k=n-\lambda_n}^{n}
S_k(f;x), \qquad n\in\mathbb{N},
\end{equation}
where $S_k(f)$ denotes the $k$th Walsh partial sum of $f$.

For an Orlicz function $\omega$, we denote by $L_{\omega}([0,1))$ the Orlicz
class 
\begin{equation*}
L_{\omega}([0,1)):=\Bigl\{f \text{ measurable on }[0,1): \int_0^1
\omega(|f(x)|)\,dx<\infty\Bigr\}.
\end{equation*}
Recall that for every Orlicz function, the quotient $\omega(t)/t$ is
nondecreasing on $(0,\infty)$; see, for example, \cite[Ch.~1]%
{Krasnoselskii-Rutickii}.

Our main result identifies the sharp window criterion for almost everywhere
convergence and shows that, on the divergence side, the obstruction to
convergence already appears in substantially smaller Orlicz clases.

\begin{theorem}
\label{thm:main} Let $\lambda=\{\lambda_n\}_{n\ge1}$ satisfy %
\eqref{eq:lambda-basic}.

\begin{itemize}
\item[(a)] If 
\begin{equation}  \label{eq:LT-good}
n=O(\lambda_n) \qquad (n\to\infty),
\end{equation}
then for every $f\in L^{1}([0,1))$, 
\begin{equation*}
V_n^{(\lambda)}(f;x)\longrightarrow f(x) \qquad \text{for almost every }%
x\in[0,1).
\end{equation*}

\item[(b)] Let $\omega :[0,\infty )\rightarrow \lbrack 0,\infty )$ be an
Orlicz function\footnote{$\omega $ is a continuous convex function on $%
[0,\infty )$ such that $\omega (0)=0$, $\omega (u)>0$ for $u>0$, and $\omega
(u)\rightarrow \infty $ as $u\rightarrow \infty $.}, and assume that 
\begin{equation}
\omega (t)=o\bigl(t\sqrt{\log t}\bigr)\qquad (t\rightarrow \infty ).
\label{eq:orlicz-subcritical}
\end{equation}%
If 
\begin{equation}
\sup_{n\geq 1}\frac{n}{\lambda _{n}}=\infty ,  \label{eq:LT-bad}
\end{equation}%
then there exists a function $f\in L_{\omega }([0,1))$ such that 
\begin{equation}
\limsup_{n\rightarrow \infty }|V_{n}^{(\lambda )}(f;x)|=\infty \qquad \text{%
for every }x\in \lbrack 0,1).  \label{eq:main-div}
\end{equation}
\end{itemize}
\end{theorem}

\begin{remark}
	 The analogue of Theorem 1.1(b) for the trigonometric system remains open. On the other hand, almost everywhere divergence in $L^1$ was established by Tandori \cite{Tandori-VP}. Part (a) of Theorem 1.1 is a Walsh analogue of Leindler's theorem \cite{Leindler} for the trigonometric system.
	\end{remark}

\section{Preliminaries and an auxiliary block polynomial}

We write $\mathbb{N}_{0}:=\mathbb{N}\cup \{0\}$ and identify $[0,1)$ with
the dyadic group in the standard way. For $n\in \mathbb{N}_{0}$ we denote by 
\begin{equation}
n=\sum_{j=0}^{\infty }n_{j}2^{j},\qquad n_{j}\in \{0,1\},  \label{eq:binary}
\end{equation}%
its binary expansion, and we write 
\begin{equation*}
|n|:=\max \{j\in \mathbb{N}_{0}:n_{j}=1\}\qquad (n\geq 1).
\end{equation*}%
The dyadic sum of nonnegative integers is defined by 
\begin{equation*}
n\oplus m:=\sum_{j=0}^{\infty }|n_{j}-m_{j}|2^{j}.
\end{equation*}%
Let $r_{j}(x):=(-1)^{\lfloor 2^{j+1}x\rfloor }$, $j\in \mathbb{N}_{0}$, be
the Rademacher functions. The Walsh--Paley system is defined by 
\begin{equation*}
w_{n}(x):=\prod_{j=0}^{\infty }r_{j}(x)^{n_{j}},\qquad n\in \mathbb{N}_{0}.
\end{equation*}%
If $f\in L^{1}([0,1))$, then its Walsh coefficients are 
\begin{equation*}
\widehat{f}(n):=\int_{0}^{1}f(x)w_{n}(x)\,dx,
\end{equation*}%
and the $n$th Walsh partial sum is 
\begin{equation*}
S_{n}(f;x):=\sum_{k=0}^{n-1}\widehat{f}(k)w_{k}(x),\qquad n\in \mathbb{N}.
\end{equation*}%
For a Walsh polynomial $P$, we write 
\begin{equation*}
\mathrm{Spec}(P):=\{k\in \mathbb{N}_{0}:\widehat{P}(k)\neq 0\}.
\end{equation*}

The next proposition is the basic building block of the divergence
construction.

\begin{proposition}
\label{prop:block} Let $m,\gamma\in\mathbb{N}$ with $2\gamma<m$. Define 
\begin{equation}  \label{eq:def-mu}
\mu(v,j):=2^{m-\gamma}v+2^{m-2\gamma}(v\oplus 2^j), \qquad 0\le
v<2^{\gamma}, \quad 0\le j<\gamma,
\end{equation}
and 
\begin{equation}  \label{eq:def-P}
P_{m,\gamma}(x):=\frac{1}{\sqrt{\gamma}}\sum_{v=0}^{2^{\gamma}-1}%
\sum_{j=0}^{\gamma-1} w_{\mu(v,j)}(x).
\end{equation}
Then the following assertions hold.

\begin{itemize}
\item[(i)] $\|P_{m,\gamma}\|_{L^{1}([0,1))}\le1$.

\item[(ii)] $\|P_{m,\gamma}\|_{L^{\infty}([0,1))}\le 2^{\gamma}\sqrt{\gamma}$%
.

\item[(iii)] 
\begin{equation*}
\mathrm{Spec}(P_{m,\gamma})\subset [2^{m-2\gamma},2^m)\cap 2^{m-2\gamma}%
\mathbb{N}_0.
\end{equation*}

\item[(iv)] For every $x\in[0,1)$ there exists an integer $\ell(x)$ such
that 
\begin{equation*}
\ell(x)\in [2^{m-\gamma},2^m)\cap 2^{m-2\gamma}\mathbb{N}
\end{equation*}
and 
\begin{equation*}
|S_{\ell(x)}(P_{m,\gamma};x)|\ge \frac14\sqrt{\gamma}.
\end{equation*}
\end{itemize}
\end{proposition}

\begin{proof}
Let 
\begin{equation*}
v=\sum_{k=0}^{\gamma -1}\nu _{k}2^{k},\qquad \nu _{k}\in \{0,1\}.
\end{equation*}%
Since the binary digit blocks of $2^{m-\gamma }v$ and $2^{m-2\gamma
}(v\oplus 2^{j})$ are disjoint, we have 
\begin{equation*}
w_{\mu (v,j)}(x)=w_{2^{m-\gamma }v}(x)\,w_{2^{m-2\gamma }(v\oplus 2^{j})}(x).
\end{equation*}%
A direct computation gives 
\begin{equation*}
w_{2^{m-\gamma }v}(x)=\prod_{k=0}^{\gamma -1}r_{m-\gamma +k}(x)^{\nu _{k}}
\end{equation*}%
and 
\begin{equation*}
w_{2^{m-2\gamma }(v\oplus 2^{j})}(x)=r_{m-2\gamma +j}(x)\prod_{k=0}^{\gamma
-1}r_{m-2\gamma +k}(x)^{\nu _{k}}.
\end{equation*}%
Therefore 
\begin{eqnarray*}
w_{\mu (v,j)}(x) &=&r_{m-2\gamma +j}(x)\prod_{k=0}^{\gamma -1}\bigl(%
r_{m-\gamma +k}(x)r_{m-2\gamma +k}(x)\bigr)^{\nu _{k}} \\
&=&r_{m-2\gamma +j}(x)c\left( x\right) ,
\end{eqnarray*}%
where%
\begin{equation*}
c\left( x\right) :=\prod_{k=0}^{\gamma -1}\bigl(r_{m-\gamma
+k}(x)r_{m-2\gamma +k}(x)\bigr)^{\nu _{k}}\in \left\{ -1,1\right\} .
\end{equation*}%
Summing over $v$ yields 
\begin{equation}
P_{m,\gamma }(x)=\frac{1}{\sqrt{\gamma }}\sum_{j=0}^{\gamma -1}r_{m-2\gamma
+j}(x)\prod_{k=0}^{\gamma -1}\bigl(1+r_{m-\gamma +k}(x)r_{m-2\gamma +k}(x)%
\bigr).  \label{eq:P-factor}
\end{equation}

Let 
\begin{equation*}
E_{m,\gamma}:=\bigl\{x\in[0,1): r_{m-\gamma+k}(x)=r_{m-2\gamma+k}(x)\text{
for }0\le k<\gamma\bigr\}.
\end{equation*}
Then \eqref{eq:P-factor} becomes 
\begin{equation}  \label{eq:P-on-E}
P_{m,\gamma}(x)=\frac{2^{\gamma}}{\sqrt{\gamma}}\mathbf{1}%
_{E_{m,\gamma}}(x)\sum_{j=0}^{\gamma-1} r_{m-2\gamma+j}(x).
\end{equation}
Assertion~(ii) follows immediately from \eqref{eq:P-on-E}.

To prove~(i), note that on $E_{m,\gamma}$ the vector 
\begin{equation*}
\bigl(r_{m-2\gamma}(x),\dots,r_{m-\gamma-1}(x)\bigr)
\end{equation*}
assumes each value in $\{\pm1\}^{\gamma}$ on a set of measure $2^{-2\gamma}$%
. Hence, by \eqref{eq:P-on-E}, 
\begin{align*}
\|P_{m,\gamma}\|_{L^{1}([0,1))} &=\frac{2^{\gamma}}{\sqrt{\gamma}}%
\int_{E_{m,\gamma}} \left|\sum_{j=0}^{\gamma-1} r_{m-2\gamma+j}(x)\right|\,dx
\\
&=\frac{1}{2^{\gamma}\sqrt{\gamma}}\sum_{\varepsilon\in\{\pm1\}^{\gamma}}
|\varepsilon_1+\cdots+\varepsilon_{\gamma}| \\
&\le \frac{1}{\sqrt{\gamma}}\left(\frac{1}{2^{\gamma}}\sum_{\varepsilon\in\{%
\pm1\}^{\gamma}} |\varepsilon_1+\cdots+\varepsilon_{\gamma}|^2\right)^{1/2}
\le 1.
\end{align*}
This proves~(i).

Since every $\mu(v,j)$ is a multiple of $2^{m-2\gamma}$, the right-hand side
of \eqref{eq:def-P} shows that 
\begin{equation*}
\mathrm{Spec}(P_{m,\gamma})\subset 2^{m-2\gamma}\mathbb{N}_0.
\end{equation*}
Moreover, 
\begin{equation*}
2^{m-2\gamma}\le \mu(v,j)<2^m,
\end{equation*}
which proves~(iii).

Fix $x\in[0,1)$. Choose $\sigma(x)\in\{\pm1\}$ that occurs at least $\gamma
/2$ times among 
\begin{equation*}
r_{m-2\gamma}(x),\dots,r_{m-\gamma-1}(x),
\end{equation*}
and put 
\begin{equation*}
J(x):=\{0\le j<\gamma: r_{m-2\gamma+j}(x)=\sigma(x)\}, \qquad \#J(x)\ge 
\frac{\gamma}{2},
\end{equation*}
\begin{equation*}
v(x):=\sum_{j\in J(x)} 2^j.
\end{equation*}
Define 
\begin{equation*}
\ell_1(x):=2^{m-\gamma}v(x), \qquad
\ell_2(x):=2^{m-\gamma}v(x)+2^{m-2\gamma}v(x).
\end{equation*}
Since $v(x)\ge1$, we have 
\begin{equation*}
2^{m-\gamma}\le \ell_1(x)<\ell_2(x)<2^m, \qquad \ell_1(x),\ell_2(x)\in
2^{m-2\gamma}\mathbb{N}.
\end{equation*}
For fixed $v=v(x)$ we have 
\begin{equation*}
\mu(v,j)=2^{m-\gamma}v+2^{m-2\gamma}(v\oplus 2^j).
\end{equation*}
Because $v\oplus 2^j<v$ if and only if the $j$th binary digit of $v$ equals $%
1$, it follows that 
\begin{equation*}
\mu(v(x),j)\in [\ell_1(x),\ell_2(x)) \qquad \Longleftrightarrow \qquad j\in
J(x).
\end{equation*}
Consequently, 
\begin{align*}
&S_{\ell_2(x)}(P_{m,\gamma};x)-S_{\ell_1(x)}(P_{m,\gamma};x) \\
&\qquad = \frac{1}{\sqrt{\gamma}}\sum_{j\in J(x)} w_{\mu(v(x),j)}(x) =\frac{%
c(x)}{\sqrt{\gamma}}\sum_{j\in J(x)} r_{m-2\gamma+j}(x)
\end{align*}
with a unimodular factor $c(x)\in\{\pm1\}$ independent of $j$. Since $%
r_{m-2\gamma+j}(x)=\sigma(x)$ for every $j\in J(x)$ and $|c(x)|=1$, it
follows that 
\begin{equation*}
\left|S_{\ell_2(x)}(P_{m,\gamma};x)-S_{\ell_1(x)}(P_{m,\gamma};x)\right|= 
\frac{\#J(x)}{\sqrt{\gamma}}\ge \frac12\sqrt{\gamma}.
\end{equation*}
Since for arbitrary real numbers $A$ and $B$ one has 
\begin{equation*}
\max\{|A|,|A+B|\}\ge \frac12|B|,
\end{equation*}
we conclude that at least one of $\ell_1(x)$ and $\ell_2(x)$ satisfies 
\begin{equation*}
|S_{\ell(x)}(P_{m,\gamma};x)|\ge \frac14\sqrt{\gamma}.
\end{equation*}
This proves~(iv).
\end{proof}

\begin{corollary}
\label{cor:VP-equals-partial} Let $m,\gamma\in\mathbb{N}$ with $2\gamma<m$,
let $P_{m,\gamma}$ be as in Proposition~\ref{prop:block}, and let $\ell(x)$
be chosen as in Proposition~\ref{prop:block}(iv). If $\lambda_{%
\ell(x)}<2^{m-2\gamma}$, then 
\begin{equation*}
V_{\ell(x)}^{(\lambda)}(P_{m,\gamma};x)=S_{\ell(x)}(P_{m,\gamma};x).
\end{equation*}
In particular, 
\begin{equation*}
|V_{\ell(x)}^{(\lambda)}(P_{m,\gamma};x)|\ge \frac14\sqrt{\gamma}.
\end{equation*}
\end{corollary}

\begin{proof}
Put $B:=2^{m-2\gamma}$. By Proposition~\ref{prop:block}(iii), every Walsh
frequency of $P_{m,\gamma}$ is a multiple of $B$. Hence 
\begin{equation}  \label{eq:block-constant}
S_k(P_{m,\gamma};x)=S_{qB}(P_{m,\gamma};x) \qquad ((q-1)B<k\le qB).
\end{equation}
Since $\ell(x)\in B\mathbb{N}$, we may write $\ell(x)=qB$. If $%
\lambda_{\ell(x)}<B$, then every integer 
\begin{equation*}
k\in [\ell(x)-\lambda_{\ell(x)},\ell(x)]
\end{equation*}
belongs to the same block $((q-1)B,qB]$, and therefore %
\eqref{eq:block-constant} implies 
\begin{equation*}
V_{\ell(x)}^{(\lambda)}(P_{m,\gamma};x)=\frac{1}{\lambda_{\ell(x)}+1}%
\sum_{k=\ell(x)-\lambda_{\ell(x)}}^{\ell(x)}
S_k(P_{m,\gamma};x)=S_{\ell(x)}(P_{m,\gamma};x).
\end{equation*}
The lower bound now follows from Proposition~\ref{prop:block}(iv).
\end{proof}

\section{Proof of Theorem~\protect\ref{thm:main}}

\begin{proof}[Proof of Theorem~\protect\ref{thm:main}(a)]
Assume \eqref{eq:LT-good}. Then there exists $c>0$ such that $\lambda_n\ge
cn $ for all sufficiently large $n$. Define the maximal operator 
\begin{equation*}
M_{\lambda}f(x):=\sup_{n\ge1} |V_n^{(\lambda)}(f;x)|.
\end{equation*}
For all sufficiently large $n$, 
\begin{align*}
|V_n^{(\lambda)}(f;x)| &\le \frac{1}{\lambda_n+1}\sum_{k=n-\lambda_n}^{n}
|S_k(f;x)| \\
&\le \frac{n}{\lambda_n+1}\cdot \frac{1}{n}\sum_{k=1}^{n} |S_k(f;x)| \le C%
\frac{1}{n}\sum_{k=1}^{n} |S_k(f;x)|
\end{align*}
with a constant $C$ independent of $n$ and $x$. Therefore, fore some $n_0
\in \mathbb{N}$ 
\begin{equation}  \label{eq:maximal-dominated}
M_{\lambda}f(x)\le C\sigma^{*}f(x)+\max_{1\le n<n_0} |V_n^{(\lambda)}(f;x)|,
\end{equation}
where 
\begin{equation*}
\sigma^{*}f(x):=\sup_{n\ge1} \frac{1}{n}\sum_{k=1}^{n} |S_k(f;x)|.
\end{equation*}
Rodin proved that 
\begin{equation}  \label{eq:Rodin}
\|\sigma^{*}f\|_{L^{1,\infty}([0,1))}\le C\|f\|_{L^{1}([0,1))} \qquad (f\in
L^{1}([0,1)))
\end{equation}
for the Walsh system \cite{Rodin}. Moreover, the second term on the
right-hand side of \eqref{eq:maximal-dominated} is the maximum of finitely
many fixed bounded operators and is therefore bounded from $L^{1}([0,1))$ to
itself. Hence \eqref{eq:maximal-dominated} implies the weak-type estimate 
\begin{equation}  \label{eq:Mlambda-weak}
\|M_{\lambda}f\|_{L^{1,\infty}([0,1))}\le C\|f\|_{L^{1}([0,1))}.
\end{equation}

Let $P$ be a Walsh polynomial. Then $S_k(P)=P$ for all $k$ larger than the
degree of $P$, and therefore 
\begin{equation*}
V_n^{(\lambda)}(P;x)\longrightarrow P(x) \qquad (n\to\infty)
\end{equation*}
for every $x\in[0,1)$: the number of indices $k$ in the averaging interval
for which $S_k(P)\ne P$ is bounded independently of $n$, while $%
\lambda_n\to\infty$ by \eqref{eq:LT-good}.

Now let $f\in L^{1}([0,1))$, and choose Walsh polynomials $P_j$ such that $%
\|f-P_j\|_{L^{1}([0,1))}\to0$. Since $V_n^{(\lambda)}(P_j;x)\to P_j(x)$
everywhere, for every $\eta>0$, 
\begin{equation*}
\begin{aligned} \bigl\{x: \limsup_{n\to\infty}
|V_n^{(\lambda)}(f;x)-f(x)|>2\eta\bigr\} &\subset \{x:
M_{\lambda}(f-P_j)(x)>\eta\} \\ &\qquad \cup \{x: |f(x)-P_j(x)|>\eta\}.
\end{aligned}
\end{equation*}
Taking measures and using \eqref{eq:Mlambda-weak}, we obtain 
\begin{equation*}
\begin{aligned} \mu\bigl\{x: \limsup_{n\to\infty}
|V_n^{(\lambda)}(f;x)-f(x)|>2\eta\bigr\} &\le
\frac{C}{\eta}\|f-P_j\|_{L^{1}([0,1))} \\ &\qquad +
\frac{1}{\eta}\|f-P_j\|_{L^{1}([0,1))}. \end{aligned}
\end{equation*}
Letting $j\to\infty$ gives the desired almost everywhere convergence.
\end{proof}

\begin{proof}[Proof of Theorem~\protect\ref{thm:main}(b)]
Fix an Orlicz function $\omega$ satisfying \eqref{eq:orlicz-subcritical},
and assume \eqref{eq:LT-bad}.

For each $a\in\mathbb{N}$ and each integer $q\ge1$, set 
\begin{equation*}
\delta(a,q):=\min\left\{2^{-a},\frac{2^{2q}}{4^a\,\omega(2^{2q})}\right\}.
\end{equation*}
Because of \eqref{eq:orlicz-subcritical}, 
\begin{equation*}
\frac{2^{2q}\sqrt{q}}{\omega(2^{2q})}\longrightarrow \infty \qquad
(q\to\infty),
\end{equation*}
and therefore, for each fixed $a$, 
\begin{equation*}
\delta(a,q)\sqrt{q}\longrightarrow \infty \qquad (q\to\infty).
\end{equation*}
Choose recursively a sequence $\{\gamma_a\}_{a\ge1}\subset\mathbb{N}$ such
that, with 
\begin{equation*}
\delta_a:=\delta(a,\gamma_a),
\end{equation*}
we have 
\begin{equation}  \label{eq:gamma-choice}
\delta_a\sqrt{\gamma_a}>16\left(a+\sum_{k=1}^{a-1} \delta_k 2^{\gamma_k}%
\sqrt{\gamma_k}\right) \qquad (a\ge1).
\end{equation}

Next, using \eqref{eq:LT-bad}, choose an increasing sequence $%
\{N_a\}_{a\ge1}\subset\mathbb{N}$ such that, with $m_a:=|N_a|$, 
\begin{equation}  \label{eq:Na-choice}
\frac{N_a}{\lambda_{N_a}}>2^{2\gamma_a+1}
\end{equation}
and 
\begin{equation}  \label{eq:ma-separation}
m_a>m_{a-1}+2\gamma_a \qquad (a\ge2).
\end{equation}
Indeed, since the ratios $N/\lambda_N$ are unbounded, for every prescribed
threshold there are arbitrarily large integers $N$ satisfying that
threshold, so the recursive choice is possible.

Now \eqref{eq:Na-choice} implies 
\begin{equation}  \label{eq:lambda-small}
\lambda_{N_a}<\frac{N_a}{2^{2\gamma_a+1}}<2^{m_a-2\gamma_a}.
\end{equation}
Let 
\begin{equation*}
W_a:=P_{m_a,\gamma_a}, \qquad f(x):=\sum_{a=1}^{\infty} \delta_a W_a(x).
\end{equation*}
Since $\delta_a\le 2^{-a}$ and $\|W_a\|_{L^{1}([0,1))}\le1$ by Proposition~%
\ref{prop:block}(i), the series converges absolutely for almost every $x$
and in $L^{1}([0,1))$, so $f\in L^{1}([0,1))$.

We now prove that $f\in L_{\omega}([0,1))$. By Proposition~\ref{prop:block}%
(ii), 
\begin{equation*}
|W_a(x)|\le 2^{\gamma_a}\sqrt{\gamma_a}\le 2^{2\gamma_a} \qquad (x\in[0,1)).
\end{equation*}
Since $\sum_{a=1}^{\infty} 2^{-a}=1$ and $\omega$ is convex with $%
\omega(0)=0 $, Jensen's inequality yields 
\begin{equation*}
\omega(|f(x)|)\le \sum_{a=1}^{\infty} 2^{-a}\,\omega\bigl(2^a\delta_a
|W_a(x)|\bigr) \qquad \text{for a.e. }x\in[0,1).
\end{equation*}
Now $2^a\delta_a\le1$, and the monotonicity of $\omega(t)/t$ gives 
\begin{equation*}
\omega\bigl(2^a\delta_a |W_a(x)|\bigr) \le 2^a\delta_a\,\omega(|W_a(x)|) \le
2^a\delta_a\,\frac{\omega(2^{2\gamma_a})}{2^{2\gamma_a}}\,|W_a(x)|.
\end{equation*}
Integrating and using Proposition~\ref{prop:block}(i), we obtain 
\begin{align*}
\int_0^1 \omega(|f(x)|)\,dx &\le \sum_{a=1}^{\infty} 2^{-a}\cdot
2^a\delta_a\,\frac{\omega(2^{2\gamma_a})}{2^{2\gamma_a}}\,\|W_a%
\|_{L^{1}([0,1))} \\
&\le \sum_{a=1}^{\infty} \delta_a\,\frac{\omega(2^{2\gamma_a})}{2^{2\gamma_a}%
} \le \sum_{a=1}^{\infty} 4^{-a}<\infty.
\end{align*}
Hence $f\in L_{\omega}([0,1))$.

Fix $x\in[0,1)$. By Proposition~\ref{prop:block}(iv), choose $\ell_a(x)$ so
that 
\begin{equation*}
\ell_a(x)\in [2^{m_a-\gamma_a},2^{m_a})\cap 2^{m_a-2\gamma_a}\mathbb{N}
\end{equation*}
and 
\begin{equation*}
|S_{\ell_a(x)}(W_a;x)|\ge \frac14\sqrt{\gamma_a}.
\end{equation*}
Since $\ell_a(x)<2^{m_a}\le N_a$, the monotonicity of $\lambda$ and %
\eqref{eq:lambda-small} imply 
\begin{equation*}
\lambda_{\ell_a(x)}\le \lambda_{N_a}<2^{m_a-2\gamma_a}.
\end{equation*}
Therefore Corollary~\ref{cor:VP-equals-partial} gives 
\begin{equation}  \label{eq:Wa-large}
|V_{\ell_a(x)}^{(\lambda)}(W_a;x)|\ge \frac14\sqrt{\gamma_a}.
\end{equation}

We decompose 
\begin{equation*}
V_{\ell_a(x)}^{(\lambda)}(f;x)=I_a(x)+II_a(x)+III_a(x),
\end{equation*}
where 
\begin{equation*}
I_a(x):=\delta_a V_{\ell_a(x)}^{(\lambda)}(W_a;x),
\end{equation*}
\begin{equation*}
II_a(x):=\sum_{k=1}^{a-1} \delta_k V_{\ell_a(x)}^{(\lambda)}(W_k;x), \qquad
III_a(x):=\sum_{k=a+1}^{\infty} \delta_k V_{\ell_a(x)}^{(\lambda)}(W_k;x).
\end{equation*}

We first show that $III_a(x)=0$. Indeed, by \eqref{eq:ma-separation}, for
every $k>a$, 
\begin{equation*}
m_a<m_k-2\gamma_k.
\end{equation*}
Hence 
\begin{equation*}
\ell_a(x)<2^{m_a}<2^{m_k-2\gamma_k}\le \min \mathrm{Spec}(W_k),
\end{equation*}
where the last inequality follows from Proposition~\ref{prop:block}(iii).
Therefore all Walsh partial sums of $W_k$ up to order $\ell_a(x)$ vanish,
and thus 
\begin{equation*}
III_a(x)=0.
\end{equation*}

Next we estimate $II_{a}(x)$. If $k<a$, then \eqref{eq:ma-separation}
implies $m_{k}<m_{a}-2\gamma _{a}$, so that 
\begin{equation*}
\max \mathrm{Spec}(W_{k})<2^{m_{k}}<2^{m_{a}-2\gamma _{a}}
\end{equation*}%
by Proposition~\ref{prop:block}(iii). On the other hand, 
\begin{equation*}
\ell _{a}(x)-\lambda _{\ell _{a}(x)}\geq 2^{m_{a}-\gamma _{a}}-\left(
2^{m_{a}-2\gamma _{a}}-1\right) >2^{m_{a}-2\gamma _{a}},
\end{equation*}%
since $\lambda _{\ell _{a}(x)}<2^{m_{a}-2\gamma _{a}}$ and $\lambda _{\ell
_{a}(x)}$ is an integer.  Hence every partial sum appearing in $V_{\ell
_{a}(x)}^{(\lambda )}(W_{k};x)$ already equals $W_{k}(x)$, and therefore 
\begin{equation*}
V_{\ell _{a}(x)}^{(\lambda )}(W_{k};x)=W_{k}(x).
\end{equation*}%
Using Proposition~\ref{prop:block}(ii), we conclude that 
\begin{equation}
|II_{a}(x)|\leq \sum_{k=1}^{a-1}\delta _{k}|W_{k}(x)|\leq
\sum_{k=1}^{a-1}\delta _{k}2^{\gamma _{k}}\sqrt{\gamma _{k}}<\frac{1}{16}%
\delta _{a}\sqrt{\gamma _{a}},  \label{eq:II-bound}
\end{equation}%
where the last inequality follows from \eqref{eq:gamma-choice}.

Finally, \eqref{eq:Wa-large} yields 
\begin{equation}  \label{eq:I-bound}
|I_a(x)|\ge \frac14\delta_a\sqrt{\gamma_a}.
\end{equation}
Combining \eqref{eq:II-bound}, \eqref{eq:I-bound}, and the identity $%
III_a(x)=0$, we obtain 
\begin{equation*}
|V_{\ell_a(x)}^{(\lambda)}(f;x)|\ge |I_a(x)|-|II_a(x)|\ge \frac14\delta_a%
\sqrt{\gamma_a}-\frac{1}{16}\delta_a\sqrt{\gamma_a}=\frac{3}{16}\delta_a%
\sqrt{\gamma_a}.
\end{equation*}
By \eqref{eq:gamma-choice}, the last quantity is greater than $3a$. Hence 
\begin{equation*}
|V_{\ell_a(x)}^{(\lambda)}(f;x)|>3a.
\end{equation*}
Since $a$ is arbitrary, \eqref{eq:main-div} follows for the fixed point $x$,
and therefore for every $x\in[0,1)$.
\end{proof}

\begin{question}
\label{prob:critical} Assume that 
\begin{equation*}
\sup_{n\ge1}\frac{n}{\lambda_n}=\infty.
\end{equation*}
Does the critical Orlicz class corresponding to the logarithmic-square-root
borderline guarantee almost everywhere convergence of de la Vall\'ee Poussin
means of Walsh--Fourier series?
\end{question}

\section*{Conflicts of Interest}

The authors declare that they have no conflicts of interest.

\section*{Data Availability}

Not applicable.

\end{document}